\documentclass[smallcondensed,runningheads]{svjour3}
\usepackage{amsmath,amssymb,mathtools}
\journalname{Submitted article}
\newcommand*{\id}{{\mathrm{id}}}
\newcommand*{\Ab}{{\mathbb A}}
\newcommand*{\Kb}{{\mathbb K}}
\newcommand*{\Db}{{\mathbb D}}
\newcommand*{\la}{{\langle}}
\newcommand*{\ra}{{\rangle}}
\newcommand*{\sh}{{\,\llcorner\!\llcorner\!\!\!\lrcorner\,}}

\newcommand*{\Ec}{{E}}
\newcommand*{\Gc}{{\mathrm g}}
\newcommand*{\zsf}{{\mathrm z}}
\newcommand*{\psf}{{\mathrm d}}
\newcommand*{\nsf}{{\mathrm n}}

\newcommand*{\Hb}{{\mathbb H}}

\newcommand*{\Sb}{{\mathbb S}}

\newcommand*{\conc}{{\mathrm{co}}}
\newcommand*{\shuf}{{\mathrm{sh}}}

\newcommand*{\shl}{{\text{sinhlog}}}
\newcommand*{\chl}{{\text{coshlog}}}
\newcommand*{\ind}{{\pi}}
\newcommand*{\un}{{\mathbf 1}}
\smartqed

\begin{document}

\title{Algebraic structure of stochastic expansions and efficient simulation}
\author{Kurusch Ebrahimi--Fard \and ~~~~~~~~~~~~~~~~~~~~~~ Alexander Lundervold
\and Simon J.A.~Malham \and Hans~Munthe--Kaas \and Anke Wiese}
\titlerunning{Algebraic structure of stochastic expansions}
%\authorrunning{Ebrahimi--Fard, Lundervold, Malham, Munthe--Kaas and Wiese}

\institute{Kurusch Ebrahimi--Fard \at 
Instituto de Ciencias Matem\'aticas,
Consejo Superior de Investigaciones Cient\'{i}ficas, 
C/ Nicol\'as Cabrera, no. 13-15, 28049 Madrid, Spain
\and
Alexander Lundervold \at 
Department of Mathematical Sciences,
Norwegian University of Science and Technology,
N-7491 Trondheim, Norway
\and
Simon J.A. Malham \and Anke Wiese \at
Maxwell Institute for Mathematical Sciences,
and School of Mathematical and Computer Sciences,
Heriot-Watt University, Edinburgh EH14 4AS, UK 
\and
Hans~Munthe--Kaas \at
Department of Mathematics,
University of Bergen,
Postbox 7800,
N-5020 Bergen, Norway}

\date{12th March 2012}
\voffset=10ex

\maketitle
\begin{abstract}
We investigate the algebraic structure underlying the 
stochastic Taylor solution expansion for stochastic differential systems. 
Our motivation is to construct \emph{efficient} integrators. 
These are approximations that generate strong numerical integration schemes 
that are more accurate than the corresponding stochastic Taylor approximation, 
independent of the governing vector fields and to all orders. 
The sinhlog integrator introduced by Malham \& Wiese (2009) is one example. 
Herein we: show that the natural context to study stochastic integrators 
and their properties is the convolution shuffle algebra of endomorphisms; 
establish a new whole class of efficient integrators; and then prove that, 
within this class, the sinhlog integrator generates the 
\emph{optimal} efficient stochastic integrator at all orders.
\keywords{stochastic simulation 
\and convolution shuffle algebra 
\and efficient integrators}
\subclass{60H10 \and 60H35}
%\PACS{...}
%\CRclass{...}
\end{abstract}

\section{Introduction}\label{sec:intro}
We consider the simulation of stochastic differential systems
of arbitrary order $N\in\mathbb N$. We assume for $y_t\in\mathbb R^N$
our system has the form
\begin{equation*}
y_t=y_0+\sum_{i=0}^d\int_0^tV_i(y_\tau)\,\text{d}W_\tau^i.
\end{equation*}
This system is driven by a $d$-dimensional Wiener
process $(W^1,\ldots,W^d)$ and governed by a drift vector
field $V_0$ and diffusion vector fields $V_1,\ldots,V_d$. 
We use the convention $W^0_t\equiv t$ and interpret
the stochastic integrals in the Stratonovich sense.  
Hereafter we will assume $t\in\mathbb R_+$ lies in the interval
of existence of the solution. In general, we also suppose
that the vector fields $V_i\colon\mathbb R^N\to\mathbb R^N$ for
$i=0,\ldots,d$ are sufficiently smooth and non-commuting.
We focus on solution series and their use in strong 
simulation schemes. The stochastic Taylor expansion
for the flowmap $\varphi_t\colon y_0\mapsto y_t$, taking
the data $y_0$ at time $t=0$ to the solution $y_t$ at time $t$ for 
the stochastic differential system above, is given by
(see for example Baudoin 2004 or Lyons \& Victoir 2004)
\begin{equation*}
\varphi_t=\sum_{w} J_w(t)\, V_w.
\end{equation*}
Here $w=a_1\ldots a_n$ is a word with letters
$a_1,\ldots,a_n$ chosen from the alphabet $\Ab\coloneqq\{0,1,\ldots,d\}$.
The sum is over all possible words $w$ in $\Ab^*$, 
the free monoid on $\Ab$. All the stochastic information
is encoded in the scalar random variables 
(Stratonovich integrals)
\begin{equation*}
J_w(t)\coloneqq\int_0^t\cdots\int_0^{\tau_{n-1}}
\text{d}W^{a_1}_{\tau_n}\,\cdots\,\text{d}W^{a_n}_{\tau_1}.
\end{equation*}
The partial differential operators 
$V_w\coloneqq V_{a_1}\circ\cdots\circ V_{a_n}$, constructed
by composing the vector fields, encode all the geometric information.

Strong numerical integration schemes for stochastic differential
systems are based on truncating the stochastic Taylor expansion
and applying the resulting approximate flowmap over successive small
computation subintervals spanning the global time interval of interest. 
More generally, across a computation interval $[0,t]$, for any smooth
map $f\colon\text{Diff}(\mathbb R^N)\to\text{Diff}(\mathbb R^N)$,
we can:
\begin{enumerate}
\item Construct the series $\sigma_t=f(\varphi_t)$;
\item Truncate the series $\sigma_t$ to $\hat\sigma_t$ according to a 
grading $\Gc(w)$ on the words $w$;
\item Compute $\hat\varphi_t=f^{-1}(\hat\sigma_t)$ and use 
this as the basis of a numerical scheme. 
\end{enumerate}
For example, suppose $f=\id$, the identity map. 
Then $\sigma_t$ is just the stochastic Taylor expansion $\varphi_t$, 
which we split according to the grading 
$\Gc(w)$ as follows
\begin{equation*}
\varphi_t=\sum_{\Gc(w)\leqslant n}J_w\,V_w+\sum_{\Gc(w)\geqslant n+1}J_w\,V_w,
\end{equation*}
for $n\in\mathbb N$.
A stochastic Taylor numerical scheme of strong order $n/2$ would 
be the first term on the right shown; the remainder is the last term. 
The grading $\Gc(w)$ here is determined by the variance of the  
stochastic integrals $J_w$; zero letters in $w$ contribute a count of one
while non-zero letters contribute a count of one-half towards $\Gc(w)$.
An important technicality is to include in the integrator, the 
\emph{expectation} of the terms in the remainder at leading order. 
This is because not including them would decrease the
expected global order by one-half (an explanation can be found in Buckwar,
Malham \& Wiese 2012 or Malham \& Wiese 2009). In other words, 
a stochastic Taylor integrator of strong order $n/2$ is 
\begin{equation*}
\hat\varphi_t=\sum_{\Gc(w)\leqslant n}J_w\,V_w+\sum_{\Gc(w)=n+1}\bar{\Ec}(J_w)\,V_w,
\end{equation*}
where the expectations of the $J_w$, here denoted $\bar{\Ec}(J_w)$, 
are known analytically.
The \emph{Euler--Maruyama} and \emph{Milstein} numerical methods correspond to 
the cases $n=1$ and $n=2$, respectively, applied on successive
computation subintervals with the vector fields evaluated on the
initial data on each subinterval. Stochastic Runge--Kutta
methods are constructed by replacing the partial differential operators 
$V_w$ by finite differences. 

Another example is $f=\text{log}$ so that $\sigma_t=\text{log}\varphi_t$. 
This is the exponential Lie series 
which is the basis of the Castell--Gaines method (see Castell \& Gaines~1995, 1996;
also see Azencott 1982, Ben Arous 1989 and Castell 1993).
Truncating the exponential Lie series to $\hat\sigma_t$ generates a Lie polynomial
in the Lie algebra of vector fields. We assume we can
suitably approximately simulate the multiple integrals $J_w$ retained
in the truncation (more on this presently). Hence 
$\hat\varphi_t=\text{exp}\hat\sigma_t$
and our approximation $\hat y_t$ to the solution $y_t$ 
across $[0,t]$ can be generated as follows. For a given realization
of the $J_w$ terms retained, we simply solve the
ordinary differential system $u'=\hat\sigma_t\circ u$ for $u=u(\tau)$,
for $\tau\in[0,1]$ and $u(0)=y_0$. Though we might achieve this
analytically, more often one has to use a suitably accurate ordinary
differential integrator. In either case we have $u(1)\approx\hat y_t$. 

We measure the accuracy of a strong order integrator by 
the root-mean-square of its local remainder 
\begin{equation*}
r_t\coloneqq\varphi_t-\hat\varphi_t.
\end{equation*}
More precisely, we measure $\|r_t\circ y_0\|_{L^2}$ for each $y_0\in\mathbb R^N$,
i.e.\/ the square-root of the expectation of the Euclidean norm of $r_t\circ y_0$. 
Let us now clarify an important issue.
The order of a strong numerical method is determined 
by the set of multiple Wiener integrals simulated and included. 
Since the set of all multiple (Stratonovich) Wiener integrals is generated
by those based on Lyndon words (see Reutenauer 1993, p.~111 and Gaines 1994),
we need only simulate the multiple Wiener integrals indexed by Lyndon
words. The other multiple Wiener integrals of that order can be
computed by linear combinations of products of the appropriate 
Lyndon word multiple integrals of that order or less. 
However Lyndon word multiple integrals of the same order 
cannot be generated as such (or from each other). Hence
more correctly, the order of a method is determined 
by the set of Lyndon word multiple Wiener integrals included. 
Consequently, in general with this multiple integral set, 
a more accurate method can only have a better error constant 
and an improvement in order scaling is not possible.
Throughout this article we assume we can suitably 
approximate/simulate the Lyndon word multiple Wiener integrals 
up to the order required.
The bulk of computational effort in higher order strong simulation methods 
is devoted to this task, though there has been some recent advances 
on this front; see Wiktorsson~(2001), Lyons and Victoir (2004),
Levin \& Wildon (2008) and Malham \& Wiese~(2011).

Castell \& Gaines (1995, 1996) proved that their simulation
method of strong order one-half was \emph{asymptotically efficient} in
the sense of Newton (1991): it ``minimizes the leading coefficient
in the expansion of the mean-square errors as power series in
the sample step-size''. This property extends to their strong order
one method when the diffusion vector fields commute.
However, Malham \& Wiese (2009) demonstrated
that when the stochastic differential system above is driven by
a multi-dimensional driving Wiener process and the governing 
diffusion vector fields do not commute, then a strong numerical
simulation based on the exponential Lie series is not 
asymptotically efficient (independent of the vector fields);
also see Lord, Malham \& Wiese (2008). Malham \& Wiese (2009) 
proved, in the absence of drift, that a strong simulation
method generated by taking the sinhlog of the flowmap, truncating 
the resulting series and then taking the inverse sinhlog, is
\emph{efficient} to all orders. This means that 
the error of the sinhlog integrator is always smaller than
the error of the corresponding stochastic Taylor integrator
in the mean-square sense, independent of the vector fields. 
In this paper we:
\begin{enumerate}
\item Show that the natural context to study stochastic integrators
and their properties is the convolution algebra of endomorpisms
on the Hopf shuffle algebra of words;
\item Establish a new class of efficient stochastic 
integrators using this algebraic structure 
(we include drift and grade according to word length);
\item Prove that within this class, the sinhlog integrator 
generates the \emph{optimal} efficient stochastic
integrator to all orders. By this we mean that the error of
the integrator realizes its smallest possible value compared 
to the error of the corresponding stochastic Taylor integrator,
in the mean-square sense.
\end{enumerate}

Our paper is structured as follows. In \S\ref{sec:conv}
we demonstrate the direct relation between stochastic expansions
and the convolution shuffle algebra of endomorpisms
on the Hopf shuffle algebra of words. We define an
inner product structure on the convolution shuffle algebra 
of endomorpisms that is based on the correlation measure 
between multi-dimensional stochastic processes 
in \S\ref{sec:main}. We also demonstrate some
natural useful identities and orthogonality properties of 
endomorphisms therein. 
In \S\ref{sec:proof} we prove results~2 and 3 stated above. 
Finally in \S\ref{sec:conclu} we discuss the implications of 
our results and provide some concluding remarks.

\section{Stochastic expansions and the convolution shuffle algebra}
\label{sec:conv}
We introduce the convolution shuffle algebra and show it
is the natural context to study stochastic expansions. 
For the moment, we proceed formally as the basic material 
can be found in the monograph by Reutenauer~(1993); also
see Remark~\ref{rmk:homomorphism} below.
Dropping $J$'s and $V$'s, we can represent the 
stochastic Taylor series for the flowmap by 
\begin{equation*}
\varphi=\sum_{w} w\otimes w,
\end{equation*}
which lies in the product algebra 
$\Kb\la\Ab\ra_\shuf\otimes\Kb\la\Ab\ra_\conc$
over the commutative ring $\Kb=\mathbb R$. 
We are interested in the following two Hopf algebra structures on $\Kb\la\Ab\ra$.
One has shuffle $\sh$ as product and deconcatenation $\Delta$
as coproduct ($\Kb\la\Ab\ra_\shuf$ on the left). 
The other has concatenation as product and 
deshuffle as coproduct ($\Kb\la\Ab\ra_\conc$ on the right). 
The unit, counit and antipode are the same for both 
algebras $\Kb\la\Ab\ra_\shuf$ and $\Kb\la\Ab\ra_\conc$. 
We denote the empty word $\un\in \Kb\la\Ab\ra$.
The product of two terms in $\Kb\la\Ab\ra_\shuf\otimes\Kb\la\Ab\ra_\conc$ is 
\begin{equation*}
(u \otimes x)(v\otimes y)=(u\sh v)\otimes (xy), 
\end{equation*}
where we concatenate the words on the right representing the 
composition of vector fields.
On the left, $u\sh v$ represents the sum of all possible shuffles
of the words $u$ and $v$, representing the product of two
multiple integrals. 
\begin{remark}\label{rmk:homomorphism}
The homomorphism from $\Kb\la\Ab\ra_\shuf\otimes\Kb\la\Ab\ra_\conc$ 
to the free associative algebra of vector fields is established as follows 
(see Malham \& Wiese 2009, Section 2(c)). Let $\mathbb J$ denote 
the ring generated by multiple Stratonovich integrals 
and the constant random variable $1$, with pointwise multiplication and 
addition. Also let $\mathbb V$ denote the set of all vector fields
on $\mathbb R^N$. %which is an $\mathbb R$-module over $C^\infty(\mathbb R^N)$.
The flow map $\varphi$ defined in the introduction lies in 
$\mathbb J\langle\mathbb V\rangle\cong\bigoplus_{n\geqslant 0}\mathbb J\otimes\mathbb V_n$,
where $\mathbb V_n$ is the subset of vector fields $V_w$ with $w$ of length $n$.
The linear \emph{word-to-vector field map}
$\kappa\colon\mathbb R\la\Ab\ra\rightarrow\mathbb V$
given by $\kappa\colon w\mapsto V_w$
is a concatenation homomorphism, i.e.\ $\kappa(uv)=\kappa(u)\kappa(v)$
for any $u,v\in\Ab^\ast$. And the linear \emph{word-to-integral map}
$\mu\colon\mathbb R\la\Ab\ra\rightarrow\mathbb J$ given by $\mu\colon w\mapsto J_w$
is a shuffle homomorphism, i.e.\ $\mu(u\sh v)=\mu(u)\mu(v)$ for any $u,v\in\Ab^\ast$
(see Lyons, \emph{et. al.\ } 2007, p.~35 or Reutenauer 1993, p.~56;
this also underlies our choice to use Stratonovich integrals rather than
It\^o integrals which satisfy a quasi-shuffle relation). Hence the map
$\mu\otimes\kappa\colon\Kb\la\Ab\ra_\shuf\otimes\Kb\la\Ab\ra_\conc
\rightarrow\bigoplus_{n\geqslant0}\mathbb J\otimes\mathbb V_n$
is a Hopf algebra homomorphism which naturally extends to the
free associative algebra of vector fields.
\end{remark}
Suppose we apply a polynomial or power series
function to the flow-map $\varphi$, say $f=f(\varphi)$. 
For example suppose $f$ has a simple power series 
expansion 
\begin{equation*}
f(\varphi)=\sum_{k=0}^\infty c_k\,\varphi^k
\end{equation*}
with coefficients $\{c_n\in\Kb\colon n\geqslant 0\}$. 
Then the product in $\Kb\la\Ab\ra_\shuf\otimes\Kb\la\Ab\ra_\conc$
implies that after rearrangement, we can always express the result 
in the form
\begin{equation*}
f(\varphi)=\sum_{w\in\Ab^*} (F\circ w)\otimes w,
\end{equation*}
where $F\in\text{End}(\Kb\la\Ab\ra_{\shuf})$ is given by 
($|w|$ denotes the length of the word)
\begin{equation*}
F\circ w=\sum_{k=0}^{|w|}
c_k\sum_{\substack{u_1,\ldots,u_k\in\Ab^*\\w=u_1\ldots u_k}}u_1\sh\ldots\sh u_k.
\end{equation*}
See Reutenauer (1993, p.~58) or Malham and Wiese (2009) for more details.
This suggests we can encode the action 
of any such function $f$ of the flowmap $\varphi$ by an endomorphism 
$F\in\text{End}(\Kb\la\Ab\ra_{\shuf})$. Indeed the embedding 
$\text{End}(\Kb\la\Ab\ra_{\shuf})\to\Kb\la\Ab\ra_\shuf\otimes\Kb\la\Ab\ra_\conc$
given by 
\begin{equation*}
X\mapsto\sum_{w} X(w)\otimes w,
\end{equation*}
is an algebra homomorphism for the non-commutative convolution product defined 
for all $X,Y\in\text{End}(\Kb\la\Ab\ra_{\shuf})$ by
\begin{equation*}
X\star Y=\sh\circ(X\otimes Y)\circ\Delta.
\end{equation*}
Since the coproduct $\Delta$ here is deconcatenation, which takes a 
word $w$ and produces a sum of all possible two-partitions $u\otimes v$
of $w$, we see that explicitly
\begin{equation*}
\bigl(X\star Y\bigr)(w)=\sum_{\substack{u,v\in\Ab^*\\w=uv}}X(u)\sh Y(v).
\end{equation*}
Thus, we can encode and study the structure and properties  
of functions of the flowmap through endomorphisms 
$F\in\text{End}(\Kb\la\Ab\ra_{\shuf})$, with convolution as product. 
For convenience, we henceforth denote 
\begin{equation*}
\Hb\coloneqq\text{End}(\Kb\la\Ab\ra_{\shuf}), 
\end{equation*}
the convolution shuffle algebra of endomorphisms on $\Kb\la\Ab\ra_{\shuf}$, 
which is a unital associative non-commutative $\Kb$-algebra. 
The unit $\nu$ in $\Hb$ is the composition of the unit and counit. 
Indeed $\nu$ sends non-empty words to $0$ and the empty word to itself. 
For any Hopf algebra, by definition the antipode $S$ 
is the inverse of the identity endomorphism with respect 
to the convolution. Thus we have
\begin{equation*}
S\star\id=\id\star S=\nu. 
\end{equation*}
On $\Kb\la\Ab\ra_{\shuf}$ the antipode $S\in\Hb$ is given by 
$S(a_1\ldots a_n)\coloneqq(-1)^na_n\ldots a_1$, 
i.e.~it is the sign-reversing endomorphism. 
\begin{remark}
There is a dual convolution concatenation algebra which we 
could alternatively utilize; see Reutenauer~(2009; Section~1.5).
\end{remark}
\begin{remark} 
There is natural compatibility between convolution and composition in $\Hb$. 
For an algebra homomorphism $Z\in\Hb$ one verifies that 
$Z\circ (X\star Y)=(Z\circ X)\star (Z\circ Y)$. 
For a coalgebra homomorphism $Z\in\Hb$ we have 
$(X\star Y)\circ Z=(X\circ Z)\star (Y\circ Z)$. 
\end{remark}
As the shuffle product is commutative, one can show that if 
$X,Y\in\Hb$ are algebra homomorphisms, then $X\star Y$ 
is also an algebra homomorphism. %In fact, 
The subset of algebra homomorphisms forms 
the group $\mathcal H\subset\Hb$ of $\Kb\la\Ab\ra_\shuf$-valued characters, 
with unit $\nu$. The inverse of $X\in\mathcal H$ is given by 
$X^{\star(-1)}\coloneqq X\circ S$. %For example $\id,S$ are in $\mathcal H$. 
Note $\mathcal H$ is a subgroup of the Lie group
\begin{equation*}
\mathcal G\coloneqq\bigl\{X\in \Hb\colon X(\un)=\un\bigr\}.
\end{equation*}
The corresponding Lie algebra $\mathfrak h$ of infinitesimal characters 
is a Lie subalgebra of the subalgebra
\begin{equation*}
\mathfrak g\coloneqq\bigl\{X\in\Hb\colon X(\un)=0\bigr\},
\end{equation*}
of $\Hb$. The inverse in $\mathcal G$ is given by 
$X^{\star(-1)}\coloneqq\sum_{k\geqslant 0} (\nu-X)^{\star k}$. Note
that we denote by $X^{\star k}$ the $k$-factor convolution 
product $X\star X\star\cdots\star X$
for any $X\in\Hb$. Observe that, for $X\in\mathfrak g$, 
if $w$ has length less than $k$ then 
$X^{\star k}\circ w$ returns $0$. Hence the formal
sum for $X^{\star(-1)}$ makes sense and indeed, 
note that $(\nu-X)\in\mathfrak g$ for $X\in\mathcal G$. 
In fact we observe that $\mathcal G=\{\nu\}\oplus\mathfrak g$.
See Manchon (2008), Patras \& Reutenauer 2002 
and Patras (1994) for more details.

An important endomorphism is the \emph{augmented ideal projector} 
given by 
\begin{equation*}
J\coloneqq\id-\nu.
\end{equation*}
Thus $J$ sends non-empty words to themselves, 
but the empty word to $0$, i.e.~$J\in\mathfrak g$. 
Note for example, $J^{\star k}$ takes a word $w$ with $|w|\geqslant k$ and 
creates a sum of all possible $k$-partitions
of $w$ shuffled together (with the empty word an excluded partition). 
On the other hand, $\id^{\star k}$ also includes all partitions involving
the empty word. Now observe that the endomorphism
$F$ corresponding to the function $f$ of the flowmap above, with
the power series coefficients $\{c_k\}$, is 
the endomorphism defined by the series in $\Hb$: 
\begin{equation*}
f^\star(\id)=\sum_{k=0}^\infty c_k\,\id^{\star k}.
\end{equation*}

More generally, we can consider functions on $\Hb$. For example,
for $X\in\Hb$, with $X(\un)=\epsilon\,\un$ and $\epsilon\in\Kb$, 
we can construct an endomorphism through a series
expansion about a multiple $\epsilon$ of the unit $\nu$ as follows:
\begin{equation*}
f^\star(X)=\sum_{k=0}^\infty c_k\,(X-\epsilon\,\nu)^{\star k},
\end{equation*}
where $X-\epsilon\,\nu\in\mathfrak g$. 
Of particular interest are the bijective logarithm 
$\text{log}^\star\colon\mathcal G\to\mathfrak g$ 
and exponential $\text{exp}^\star\colon\mathfrak g\to\mathcal G$ 
maps defined for any $X\in\mathfrak g$ by
\begin{equation*}
\text{log}^\star(\nu+X)=\sum_{k=1}^\infty\frac{(-1)^{k+1}}{k}X^{\star k}
\qquad\text{and}\qquad
\text{exp}^\star(X)=\sum_{k=0}^\infty\frac{1}{k!}X^{\star k}.
\end{equation*}
The sinhlog and coshlog maps are defined for $X\in\mathcal G$ by 
\begin{equation*}
\shl^\star(X)=\tfrac12\bigl(X-X^{\star(-1)}\bigr)
\qquad\text{and}\qquad
\chl^\star(X)=\tfrac12\bigl(X+X^{\star(-1)}\bigr),
\end{equation*}
also have series representations in powers of $(X-\nu)$.
These maps and their compositional inverses underlie our main result.
For all $X,Y\in\Hb$, set $h^\star$ to be
\begin{equation*}
h^\star(X,Y)\coloneqq\bigl(X^{\star 2}+Y\bigr)^{\star(1/2)},
\end{equation*}
for which the square root exists. 
Then we have the compositional inverses
\begin{equation*}
\shl^{-1}(X)=X+h^\star(X,+\nu)
\quad\text{and}\quad
\chl^{-1}(X)=X+h^\star(X,-\nu),
\end{equation*}
i.e.\/ we have $\shl^{-1}\circ\shl^\star(X)=X$
and $\chl^{-1}\circ\chl^\star(X)=X$.

To illustrate this new perspective and its 
natural connection to stochastic expansions, consider three examples. 
First, we observe that the stochastic Taylor expansion for the flowmap 
is simply the identity $\id\in\mathcal H$. Second, the sinhlog function 
considered by Malham \& Wiese (2009) is given by
\begin{equation*}
\shl^\star(\id)=\tfrac12(\id-S), 
\end{equation*}
since $S=\id^{\star(-1)}$.
In other words, applying the sinhlog function to the flowmap $\varphi$ 
in $\Kb\la\Ab\ra_\shuf\otimes\Kb\la\Ab\ra_\conc$, corresponds to applying 
$\shl^\star$ to the identity $\id\in\mathcal H$. Third, 
the Eulerian idempotent, which is a Lie idempotent from the free associative
algebra to the free Lie algebra, is given by 
\begin{equation*}
\text{log}^\star(\id)=J-\tfrac12J^{\star 2}+\tfrac13 J^{\star 3}-\cdots
+\tfrac{(-1)^{k+1}}{k}J^{\star k}+\cdots.
\end{equation*}
This is the exponential Lie series or Chen--Strichartz formula 
(see Burgunder 2009, Chen 1957, Magnus 1954, Strichartz 1987 
and Baudoin 2004). To see this we use that $J^{\star k}$ 
can be expressed as a sum over permutations with a prescribed
descent set; see Reutenauer (1993; p.~65).
Note also that since the antipode is the inverse of the identity with 
respect to the convolution product, then 
\begin{equation*}
S=\nu-J+J^{\star 2}-J^{\star 3}+\cdots.
\end{equation*}
Hence we can also express the sinhlog endomorphism as
\begin{equation*}
\shl^\star(\id)
=J-\tfrac12J^{\star 2}+\tfrac12 J^{\star 3}-\cdots
+(-1)^{k+1}\tfrac{1}{2}J^{\star k}+\cdots.
\end{equation*}
These three examples belong to the subalgebra of endomorphisms
generated by the unit $\nu$ and 
augmented ideal projector $J$.
\begin{remark}
Note that the sinhlog and coshlog endomorphisms are projectors as
$\frac12(\id\pm S)\circ\frac12(\id\pm S)=\frac12(\id\pm S)$
and $\frac12(\id\pm S)\circ\frac12(\id\mp S)=0$.
\end{remark}
We conclude this section by defining some endomorphisms and 
their properties useful in our subsequent analysis.
\begin{definition}[Reversing and sign endomorphisms]
These endomorphisms in $\Hb$ are defined for any word 
$w=a_1\ldots a_n\in\Ab^*$ as follows:
(1) Reversing endomorphism: $|S|\colon w\mapsto a_n\ldots a_1$; and
(2) Sign endomorphism: $D\colon w\mapsto (-1)^nw$.
%\begin{enumerate}
%\item Reversing endomorphism: $|S|\colon w\mapsto a_n\ldots a_1$;
%\item Sign endomorphism: $D\colon w\mapsto (-1)^nw$. 
%\end{enumerate}
\end{definition}
Note, for example, that $D\circ D\equiv\id$ and $S=D\circ|S|=|S|\circ D$. 
Observe also that since $S\in\mathcal H$ we have $|S|(u\sh v)=|S|(u)\sh |S|(v)$.

\section{Convolution shuffle algebra with expectation inner product}\label{sec:main}
We have seen that classes of endomorphisms $F\in\Hb$ correspond 
to functions of the flowmap. Our goal here is to define an appropriate inner
product on $\Hb$ and analyze its properties. This will be modelled
on the mean-square measure of an $\mathbb R^N$-valued stochastic process,
constructed as follows (hereafter the real parameter $t>0$ is fixed). 

\subsection{Expectation endomorpism}
Let $\Db^*\subset\Ab^*$ denote the free monoid of words 
on the alphabet $\Db=\{0,11,\ldots,dd\}$.
\begin{definition}[Expectation map and endomorphism]
The \emph{expectation map} is the linear map 
$\bar{\Ec}\colon\Kb\la\Ab\ra_\shuf\to\Kb$ which,
in the case $\Ab$ indexes $d$ independent Wiener processes,
is given by 
\begin{equation*}
\bar{\Ec}\colon w\mapsto\begin{cases} 
t^{\nsf(w)}/\bigl(2^{\psf(w)}\nsf(w)!\bigr),&\quad w\in\Db^*,\\
0,&\quad w\in\Ab^*\backslash\Db^*.
\end{cases}
\end{equation*}
Here $\psf(w)$ is the number of non-zero consecutive pairs in $w$ 
from the alphabet $\mathbb D$ and $\nsf(w)=\zsf(w)+\psf(w)$, where
$\zsf(w)$ counts the number of `0' letters $w$ contains.

We define the \emph{expectation endomorphism} $\Ec\in\Hb$ as
\begin{equation*}
\Ec\colon w\mapsto\bar{\Ec}(w)\,\un. 
\end{equation*}
\end{definition}
Note the endomorphisms $X-\Ec\circ X\equiv (\id-\Ec)\circ X$ 
in $\Hb$ have zero expectation. Indeed $(\id-\Ec)$ lies in the kernel
of $E$ as $E\circ E=E$.
\begin{remark}
Strictly, the expectation map $\bar{\Ec}\colon\Kb\la\Ab\ra_\shuf\to\Kb[t]$,
where $t$ is a parameter commuting with all of $\Kb\la\Ab\ra_\shuf$,
and $(\id-\Ec)\circ X$ takes values in $\Kb[t]\,\un\oplus\Kb\la\Ab\ra_\shuf$.
\end{remark}
\begin{remark}\label{rmk:expectvalues}
The values quoted for the expectation map for any $w\in\Ab^\ast$
are the expectations of the corresponding multiple
Stratonovich integrals, see
Kloeden \& Platen (1999; eqns (5.2.34), (5.7.1)) or
Buckwar \textit{et al.\/ } (2012). Briefly, 
every Stratonovich integral labelled by $w\in\Ab^*$,
is a linear combination of It\^o integrals. 
The It\^o integrals concerned, cycle through the 
set of words which consist of $w$ and all words
$u$ obtained by successively replacing any two adjacent
non-zero equal indices in $w$ by $0$. Each replacement
contributes a factor one-half to the coefficient of 
the It\^o integral in the linear combination. Since
the expectation of any It\^o integral is zero unless
all its labelling letters are $0$, the expectation
of the Stratonovich integral is given by the 
deterministic integral remaining after the replacement
process (and zero if there isn't one). 
\end{remark}

\subsection{Inner product of endomorphisms}
Let $\{\mathsf V_w\}_{w\in\Ab^\ast}$ denote a given set of indeterminate vectors
indexed by words $w\in\Ab^\ast$. We use both $(u,v)$ and 
$\mathsf V_{uv}$ to denote the inner product of the vectors 
$\mathsf V_u$ and $\mathsf V_v$. Let $\mathsf V$ denote
the infinite square matrix indexed by the words $u,v\in\Ab^*$ 
with entries $\mathsf V_{u,v}$.
\begin{definition}[Inner product]
We define the \emph{inner product} of $X,Y\in\Hb$
with respect to $\mathsf V$ to be
\begin{equation*}
\langle X, Y\rangle_\Hb\coloneqq\sum_{u,v\in\Ab^*}
\bar{\Ec}\,\bigl(X(u)\sh Y(v)\bigr)\,(u,v).
\end{equation*}
The norm of an endomorphism $X\in\Hb$
is $\|X\|_{\Hb}\coloneqq\langle X, X\rangle^{1/2}$.
\end{definition}
Let us motivate this definition and provide some 
equivalent characterizations. Suppose we apply
two separate functions to the flowmap $\varphi$
which are characterized by the endomorphisms 
$X$ and $Y$ in $\Hb$. We assume the governing
vector fields and driving Wiener processes are
given as well as some data $y_0\in\mathbb R^N$.
With a slight abuse of notation we express
the two stochastic processes $x_t$ and $y_t$ 
associated with $X$ and $Y$ as follows
\begin{equation*}
x_t=\sum_{w\in\Ab^*}X(w)\,V_w(y_0)
\qquad\text{and}\qquad
y_t=\sum_{w\in\Ab^*}Y(w)\,V_w(y_0).
\end{equation*}
Our definition is based on the $L^2$-inner product
$\langle x_t, y_t\rangle_{L^2}=\bar{\Ec}\,(x_t,y_t)$.
We would like our inner product to be independent of the 
data $y_0$, hence we replace the vector fields evaluated
at the data by the set of indeterminants $\{\mathsf V_w\}_{w\in\Ab^\ast}$.

An equivalent characterization of the inner product is 
as follows. The action of an endomorphism $X$ on a 
word can be written $X(w)=\sum_{u\in\Ab^*}\mathsf X_{w,u} u$,
for some $\Kb$-valued coefficients $\mathsf X_{w,u}$.
In other words we can represent $X\in\Hb$ by a matrix
$\mathsf X\in\Kb^\infty\times\Kb^\infty$ indexed by the
words $w\in\Ab^*$. We can order the indexing by word length
and then lexicographically within each length (for example).
Now let $\mathsf W\in\Kb^\infty\times\Kb^\infty$ denote the 
symmetric matrix of values $\bar{\Ec}\,\{u\sh v\}$ over all 
$u,v\in\Ab^*$ and $\mathsf V\in\Kb^\infty\times\Kb^\infty$ 
the symmetric matrix defined above. 
Then using our definition above, we observe that
\begin{align*}
\langle X, Y\rangle_\Hb=&\;\sum_{u,v,u',v'\in\Ab^*}
\mathsf X_{u,u'}\mathsf Y_{v,v'}\mathsf W_{u',v'}\mathsf V_{u,v}\\
=&\;\mathrm{tr}\,\bigl(\mathsf X\,\mathsf W\,\mathsf Y^{\dag}\,\mathsf V^{\dag}\bigr)\\
=&\;\mathrm{tr}\,\bigl(\mathsf V^{\frac12}\,\mathsf X\,
\mathsf W\,\mathsf Y^{\dag}\,(\mathsf V^{\frac12})^{\dag}\bigr),
\end{align*}
where $\dag$ denotes matrix transpose.
Note that both $\mathsf W$ and $\mathsf V$ are positive definite.
Hence we can view the definition of the inner product above as
defined with respect to the weights $\mathsf W$ and $\mathsf V$
where the stochastic and geometric information are respectively
encoded. All results we subsequently establish will hold independent
of $\mathsf V$. Lastly, we now note that our inner product is: 
(i) Symmetric as the shuffle product and vector inner product are commutative;
(ii) Bilinear as the endomorphisms and expectation are linear; 
(iii) Positive definite as the matrix of 
expectations $\mathsf W$ is positive definite. 
\begin{remark}
We assume that the solution to our stochastic differential system
for any data $y_0$ is $L^2$-valued on the time interval of interest. 
This is equivalent to saying that $\|\id\|_{\Hb}^2$ is finite, which
is equivalent to assuming $\mathrm{tr}\,(\mathsf W\mathsf V)$ is finite.
\end{remark}

\subsection{Graded class subspace}
We often will be concerned with endomorphisms that act on subspaces
of $\Kb\la\Ab\ra_\shuf$ selected according to a grading. 
Recall that $\Kb\la\Ab\ra_\shuf$ is a connected Hopf algebra 
graded by the length of words. 
\begin{definition}[Grading map]
This is the linear map $\Gc\colon\Kb\la\Ab\ra_\shuf\to\mathbb Z_+$ given by 
\begin{equation*}
\Gc\colon w\mapsto|w|. 
\end{equation*}
The empty word has length zero, i.e.\/ $|\un|=0$.
\end{definition}
\begin{remark}
There is another natural grading on $\Ab^\ast$ given by the variance of
the words $w\in\Ab^\ast$, as mentioned in the introduction.
This is determined by the exponent in $t$ when computing
the root-mean-square deviation from the expectation of $w$, i.e.\/
the square-root of $\bar{\Ec}\circ(w-\Ec(w))^{\sh 2}$. Briefly, 
for any word $w$, zero letters contribute a count of one 
while non-zero letters contribute a count of one-half towards
the grade value. The nuances of the two gradings are revealed
in \S\ref{sec:conclu}.
\end{remark}
\begin{definition}[Graded class subspace]
For a given $n\in\mathbb Z_+$, let $\Sb_n$ denote
the subspace of\/ $\Kb\la\Ab\ra_\shuf$ of
all words $w$ of given length $\Gc(w)=n$. We set
\begin{equation*}
\Sb_{\leqslant n}\coloneqq\bigoplus_{k\leqslant n}\Sb_k
\qquad\text{and}\qquad
\Sb_{\geqslant n}\coloneqq\bigoplus_{k\geqslant n}\Sb_k.
\end{equation*}
A subspace $\mathbb S$ is a \emph{graded class subspace} 
if, for a given $n\in\mathbb Z_+$, $\Sb=\Sb_n$, $\Sb=\Sb_{\leqslant n}$
or $\Sb=\Sb_{\geqslant n}$. For any graded class subspace $\Sb$, we 
denote by 
\begin{equation*}
\ind_\Sb\colon\Kb\la\Ab\ra_\shuf\to\Sb,
\end{equation*}
the canonical projection from $\Kb\la\Ab\ra_\shuf$ onto $\Sb$. 
\end{definition}
Hereafter we set, for any graded class subspace $\mathbb S$,
for all $X,Y\in\Hb$:
\begin{equation*}
\la X, Y\ra\coloneqq\bigl\langle X\circ\ind_{\Sb},Y\circ\ind_{\Sb}\bigr\rangle_\Hb
\qquad\text{and}\qquad 
\|X\|\coloneqq\bigl\|X\circ\ind_{\Sb}\bigr\|_\Hb.
\end{equation*}
We carefully state the subspace $\Sb$ in all instances, so no confusion should arise.

\subsection{Properties of the sinhlog and coshlog endomorphisms}
The following lemma from Malham \& Wiese (2009) is a crucial ingredient
in what follows. We restate it here and discuss 
an extension we rely upon for clarity.
\begin{lemma}[Malham \& Wiese (2009); Lemma~4.3]\label{lemma:MW2009}
For any pair $u,v\in\Ab^*$, we have 
$\Ec\,(u\sh v)\equiv \Ec\,\bigl((|S|\circ u)\sh (|S|\circ v)\bigr)$.
\end{lemma}
Though the context for this lemma was the alphabet $\Ab=\{1,\ldots,d\}$,
it in fact extends to $\Ab=\{0,1,\ldots,d\}$. The proof detailed in
Malham \& Wiese relies on two results. First, any Stratonovich integral 
can be expressed as a linear combination of It\^o integrals, as described in
Remark~\ref{rmk:expectvalues} (also see Kloeden \& Platen 1999; eqn (5.2.34)). 
Importantly, reversing the word associated with a Stratonovich integral
generates a mirror linear combination of It\^o integrals with their
respective words reversed (with the same coefficients).
Second, the expectation of the product of any two It\^o integrals 
only depends on the number of non-zero letters and the lengths of
subwords containing only $0$ letters. These characteristic quantities 
are invariant to reversing the words concerned 
(see Kloeden \& Platen 1999; Lemma 5.7.2). 
Both of these results from Kloeden \& Platen 
are stated for $\Ab=\{0,1,\ldots,d\}$.
\begin{remark}
Importantly, note that if $u$ and $v$ have the 
same length, then we can replace $|S|$ by $S$
as the result is then invariant to the sign in
the antipode $S$. 
This observation underlies the restriction
to $\Sb=\Sb_n$ when $\Ab=\{0,1,\ldots,d\}$ in 
Lemma~\ref{lemma:rev} below.
\end{remark}
We now state the main lemma of this section, outlining
properties of the principle endmorphisms we have thusfar
highlighted, in particular the endomorphisms
\begin{equation*}
\shl^\star(\id)=\tfrac12\bigl(\id-S\bigr)
\qquad\text{and}\qquad
\chl^\star(\id)=\tfrac12\bigl(\id+S\bigr).
\end{equation*} 
The results herein are used to establish the 
optimal efficiency properties of the 
sinhlog integrator in \S\ref{sec:proof}.
\begin{lemma}\label{lemma:rev}
We assume either, $\Ab=\{0,1,\ldots,d\}$ 
and $\Sb$ is the graded class subspace $\Sb=\Sb_n$, 
or, $\Ab=\{1,\ldots,d\}$ and $\Sb$ is any graded class subspace.
Then for any $X,Y\in\Hb$ (and for any $\mathsf V$) 
we have the following properties:
\begin{enumerate}
\item $\langle X, Y\rangle=\bigl\langle |S|\circ X, |S|\circ Y\bigr\rangle$;
\item $\bigl\langle |S|, |S|\bigr\rangle=\langle S, S\rangle=\langle\id, \id\rangle$;
\item $\langle\text{\emph{sinhlog}}^\star(\id),\text{\emph{coshlog}}^\star(\id)\rangle=0$;
\item $\|\id\|^2=\|\text{\emph{sinhlog}}^\star(\id)\|^2
+\|\text{\emph{coshlog}}^\star(\id)\|^2$;
\item $\langle X,E\circ Y\rangle=\langle E\circ X,E\circ Y\rangle$;
\item $\langle E\circ\id,E\circ\id\rangle=\bigl\langle E\circ|S|,E\circ|S|\bigr\rangle
=\langle E\circ S,E\circ S\rangle$;
\item $\bigl\langle E\circ\text{\emph{sinhlog}}^\star(\id),
E\circ\text{\emph{coshlog}}^\star(\id)\bigr\rangle=0$;
\item $\langle|S|,J^{\star n}\rangle=\langle\id,J^{\star n}\rangle$,
\end{enumerate}
where property~8 only holds for $\Sb=\Sb_n$, independent of the alphabet.
Property~3 indicates $\text{\emph{sinhlog}}^\star(\id)$ 
and $\text{\emph{coshlog}}^\star(\id)$ are
\emph{orthogonal} with respect to the inner product.
\end{lemma}
\begin{proof}
We establish properties 1--8 in order. 
Using the linearity properties of the expectation map
and reversing endomorphism $|S|$ we observe that 
\begin{align*}
\Ec\,\bigl(X(u)\sh Y(v)\bigr)
%=&\Ec\,\biggl(\sum_{w_1}\mathsf X_{uw_1}w_1\biggr)\sh\biggl(\sum_{w_2}\mathsf Y_{vw_2}w_2\biggr)\\
=&\sum_{w_1,w_2}\mathsf X_{uw_1}\mathsf Y_{vw_2}\,\Ec\,(w_1\sh w_2)\\
=&\sum_{w_1,w_2}\mathsf X_{uw_1}\mathsf Y_{vw_2}\,\Ec\,\bigl((|S|\circ w_1)\sh (|S|\circ w_2)\bigr)\\
%=&\Ec\,\biggl(|S|\circ\sum_{w_1}\mathsf X_{uw_1}w_1\biggr)
%\sh\biggl(|S|\circ\sum_{w_2}\mathsf Y_{vw_2}w_2\biggr)\\
=&\Ec\,\bigl((|S|\circ X)(u)\sh (|S|\circ Y)(v)\bigr),
\end{align*}
and property~1 follows. Property~2 is a special case of property~1.
Using property~2 and Lemma~\ref{lemma:MW2009} we deduce that
$4\langle\shl^\star(\id),\chl^\star(\id)\rangle
=\langle\id,\id\rangle-\langle S, S\rangle=0$, thus establishing
property~3. Property~4 follows directly when we observe that
sinhlog and coshlog additively decompose the identity, 
i.e.\/ $\id=\shl^\star(\id)+\chl^\star(\id)$. Properties~5--7
now follow from the properties of the expectation endomorphism.
Property~8 follows directly from the commutivity of the shuffle product.
\end{proof}

\section{Sinhlog and the optimal efficient integrator}\label{sec:proof}
Recall our construction of a stochastic integrator for a given 
smooth function $f\colon\text{Diff}(\mathbb R^N)\to\text{Diff}(\mathbb R^N)$
we outlined in steps 1--3 in the introduction.
In terms of the convolution shuffle algebra and corresponding
endomorphism $f^\star(\id)\in\Hb$, those steps and concepts therein, 
have natural concise translations. Proceeding through the 
construction, by direct analogy with $\hat\varphi_t$, 
a general stochastic integrator has the form
\begin{equation*}
\widehat{\id}\coloneqq f^{-1}\circ\ind_{\Sb_{\leqslant n}}\circ f^\star(\id).
\end{equation*}
The error associated with this approximation is
\begin{equation*}
R\coloneqq\id-\widehat{\id}. 
\end{equation*}
One integrator will be more accurate than another if the 
$\Hb$-norm of its associated error $\|R\|_{\Hb}$ is smaller than the other.
\begin{remark}\label{remark:gradepresfollow}
In light of our comments in the introduction that the
set of Lyndon word multiple integrals included in an
integrator determine its order, we will
restrict ourselves to functions $f^\star\colon\Hb\to\Hb$ that are
grade preserving. All the functions herein such as the logarithm,
exponential, sinhlog, coshlog or any power series in $J$
are grade preserving (their action on words  
generates linear combinations of words of the same length).
\end{remark}
\begin{remark}\label{remark:expect}
Also as mentioned in the introduction, we need to include 
the expectations of the terms in the remainder in $\Sb_{n+1}$
in the integrator.
These are in general encoded by $\Ec\circ\ind_{\Sb_{n+1}}\circ R$,
so the effective \emph{leading order terms} in the remainder are
$(\id-\Ec)\circ\ind_{\Sb_{n+1}}\circ R$.
We highlight this at the appropriate junctures.
\end{remark}
\begin{definition}[Pre-remainder]
We define the \emph{pre-remainder} $Q$ to be the remainder terms 
after applying the endomorphism $f^\star(\id)$ and then truncating
to include the terms in $\Sb_{\leqslant n}$, i.e.\/ it is
\begin{equation*}
Q\coloneqq f^\star(\id)-\ind_{\Sb_{\leqslant n}}\circ f^\star(\id).
\end{equation*}
\end{definition}
The relationship between $R$ and $Q$ plays a key role in our
subsequent analysis. % of efficient integrators.

We now examine the properties of integrators based on
the sinhlog endomorphism, indeed of a whole class of endomorphisms
to which it belongs, as well as the coshlog endomorphism. 
First we focus on the sinhlog endomorphism and re-establish
that it is an efficient integrator, in the context
of the convolution shuffle algebra. As we do so, we will see 
natural questions that emerge, motivating us to dig deeper.
\begin{definition}[Efficient integrator]
A numerical approximation to the solution
of a stochastic differential equation is an 
\emph{efficient integrator} if it generates
a strong numerical integration scheme that is more accurate in
the mean square sense
than the corresponding stochastic Taylor integration scheme of the same
strong order, independent of the governing vector fields and to all orders. 
In other words, if $R$ denotes the remainder of an
integrator, it is efficient if
for $\Sb=\Sb_{n+1}$ for any $n$, and for any $\mathsf V$:
\begin{equation*}
\|(\id-E)\circ R\|^2\leqslant\|(\id-E)\circ\id\|^2.
\end{equation*}
\end{definition}
The sinhlog endomorphism of interest has the expansion
\begin{equation*}
\shl^\star(\id)=J-\tfrac12 J^{\star2}+\tfrac12 J^{\star3}-\tfrac12 J^{\star4}+\cdots.
\end{equation*}
To construct a sinhlog integrator of strong order $n/2$, 
we start by applying the projection operator $\ind_{\Sb_{\leqslant n}}$ 
to $\shl^\star(\id)$; call the result 
\begin{equation*}
P\coloneqq\ind_{\Sb_{\leqslant n}}\circ\shl^\star(\id).
\end{equation*}
Since $J^{\star (n+1)}$ is zero for any words $w$ with $|w|<n+1$ and  
each of the terms in the expansion above is grade preserving, 
we see that
\begin{equation*}
P=\bigl(J-\tfrac12 J^{\star2}+\cdots+\tfrac12(-1)^{n+1}J^{\star n}\bigr)\circ\ind_{\Sb_{\leqslant n}}.
\end{equation*}
The pre-remainder is thus given by 
\begin{equation*}
Q=\shl^\star(\id)\circ\ind_{\Sb_{\geqslant n+1}}.
\end{equation*}
The compositional inverse of the 
sinhlog endomorphism is given by $\shl^{-1}\circ X=X+h^\star(X,+\nu)$ where
$h^\star(X,+\nu)\coloneqq(\nu+X^{\star 2})^{\star(1/2)}$ has the expansion
\begin{equation*}
h^\star(X,+\nu)=\nu+\tfrac12 X^{\star 2}-\tfrac18 X^{\star 4}+\cdots.
\end{equation*}
By definition, the error $R$ in the approximation 
$\widehat{\id}=\shl^{-1}\circ P$ is given by 
{\allowdisplaybreaks
\begin{align*}
R=&\;\id-\shl^{-1}\circ P\\ 
=&\;\shl^{-1}\circ(P+Q)-\shl^{-1}\circ P\\
=&\;Q+h^\star(P+Q,+\nu)-h^\star(P,+\nu)\\
=&\;Q+\tfrac12\bigl((P+Q)^{\star 2}-P^{\star 2}\bigr)+\cdots\\
=&\;Q+\tfrac12\bigl(P\star Q+Q\star P\bigr)+\mathcal O\bigl(Q^{\star 2}\bigr).
\end{align*}}
The leading order term in $R$ is
\begin{equation*}
Q\circ\ind_{\Sb_{n+1}}.
\end{equation*}
We justify this as follows. Consider the term $P\star Q$.
Since at leading order $\shl^\star(\id)=J$,
we see at leading order
\begin{equation*}
P\star Q=\bigl(J\circ \ind_{\Sb_{\leqslant n}}\bigr)\star\bigl(J\circ \ind_{\Sb_{\geqslant n+1}}\bigr).
\end{equation*}
This means that for some coefficients $c_w\in\Kb$,
\begin{equation*}
\sum_w \bigl((P\star Q)\circ w\bigr)\otimes w=\sum_{|w|\geqslant n+2}c_w\, w\otimes w.
\end{equation*}
This follows from the fact that the two indicator functions 
annihilate any lower order terms in the two-part deconcatenation
implied by the convolution, and that $J$ annihilates
the empty word. We also only retain the leading order
terms in $Q$ by applying $\ind_{\Sb_{n+1}}$ as shown in the 
final step above.
\begin{remark}
We shall use the big $\mathcal O$ notation such as
$\mathcal O(P\star Q)$ or $\mathcal O(Q^{\star2})$ above to 
denote endomorphisms that only generate terms involving
words that are higher order with respect to the grading 
$\Gc$ than those generated by the preceding endomorphisms.
\end{remark}
Thus using Lemma~\ref{lemma:rev}(3) with $\Sb=\Sb_{n+1}$ we have 
\begin{equation*}
\|\id\|^2=\|Q\|^2+\|\chl^\star(\id)\|^2.
\end{equation*}
Now using $\id=\shl^\star(\id)+\chl^\star(\id)$ and 
Lemma~\ref{lemma:rev}(5,7) we deduce that
\begin{equation*}
\bigl\|(\id-E)\circ\id\bigr\|^2=\bigl\|(\id-E)\circ Q\bigr\|^2
+\bigl\|(\id-E)\circ\chl^\star(\id)\bigr\|^2.
\end{equation*}
Recall from Remark~\ref{remark:expect} that we need to include in our 
integrators the expectation of the leading order terms in the remainder. 
This means the remainders for the stochastic Taylor expansion and 
sinhlog integrators are $(\id-E\circ\id)\circ\ind_{\Sb_{n+1}}$
and $(Q-E\circ Q)\circ\ind_{\Sb_{n+1}}$, respectively, since 
at leading order $R=Q\circ\ind_{\Sb_{n+1}}$.
We have thus established that integrators based
on the sinhlog endomorphism are at least as accurate
as corresponding stochastic Taylor integrators at
leading order, i.e.\/ they are efficient.

We can prove the following extension of Corollary~4.2
in Malham \& Wiese (2009) to the alphabet $\Ab^*=\{0,1,\ldots,d\}$;
indeed assume this to be the alphabet hereafter.
\begin{lemma}
With $\Sb=\Sb_{n+1}$, if $n$ is \emph{odd}, 
then the error of the sinhlog integrator is optimal; it realizes its
smallest possible value compared to the error of the corresponding 
stochastic Taylor integrator. 
\end{lemma}
\begin{proof}
To prove this, order by order we 
perturb the sinhlog integrator as follows. For a fixed $n\in\mathbb N$ 
perturb the coefficient of $J^{\star(n+1)}$ in the sinhlog expansion so that  
\begin{equation*}
\shl^\star_\epsilon(\id)=J+\tfrac12\sum_{k=2}^\infty(-1)^{k+1}J^{\star k}
+\epsilon J^{\star(n+1)},
\end{equation*}
where $\epsilon\in\mathbb R$ is a parameter. Then repeating the procedure
above, we have
\begin{equation*}
Q_\epsilon
=\bigl(\shl^\star(\id)+\epsilon J^{\star(n+1)}\bigr)\circ\ind_{\Sb_{\geqslant n+1}}.
\end{equation*}
Since the leading order behaviour of the inverse of $\shl^\star$ is 
unaffected, we have at leading order $R_\epsilon=Q_\epsilon$. As above,
truncate $Q_\epsilon$ with $\ind_{\Sb_{n+1}}$. Then with $\Sb=\Sb_{n+1}$ 
we see 
\begin{align*}
\|\id\|^2
=&\;\bigl\langle Q_\epsilon+\chl^\star(\id)-\epsilon J^{\star(n+1)}, 
Q_\epsilon+\chl^\star(\id)-\epsilon J^{\star(n+1)}\bigr\rangle\\
=&\;\|Q_\epsilon\|^2
+2\bigl\langle Q_\epsilon,\chl^\star(\id)-\epsilon J^{\star(n+1)}\bigr\rangle
+\bigl\|\chl^\star(\id)-\epsilon J^{\star(n+1)}\bigr\|^2\\
=&\;\|Q_\epsilon\|^2+\|\chl^\star(\id)\|^2
-2\epsilon\bigl\langle\shl^\star(\id),J^{\star(n+1)}\bigr\rangle
-\epsilon^2\bigl\|J^{\star(n+1)}\bigr\|^2\\
=&\;\|Q_\epsilon\|^2+\|\chl^\star(\id)\|^2
-\epsilon\bigl\langle\id-S,J^{\star(n+1)}\bigr\rangle
-\epsilon^2\bigl\|J^{\star(n+1)}\bigr\|^2,
\end{align*}
where in the penultimate step
we used the orthogonality property of sinhlog and coshlog. 
When we include the expectations of the leading order terms 
in the remainder (cf.\/ Remark~\ref{remark:expect}) this
relation becomes
\begin{align*}
\bigl\|(\id-E)\circ\id\bigr\|^2
=&\;\bigl\|(\id-E)\circ Q_\epsilon\bigr\|^2+\bigl\|(\id-E)\circ\chl^\star(\id)\bigr\|^2\\
&\;-\epsilon\bigl\langle(\id-E)\circ(\id-S),(\id-E)\circ J^{\star(n+1)}\bigr\rangle\\
&\;-\epsilon^2\bigl\|(\id-E)\circ J^{\star(n+1)}\bigr\|^2.
\end{align*}
The linear term in $\epsilon$ is zero 
if $n$ is odd, by Lemma~\ref{lemma:rev}(8). Hence in this case, 
the mean-square excess, the terms on the right other than $\|Q_\epsilon\|^2$, is
optimized when $\epsilon=0$. 
\end{proof}
\begin{remark}\label{rmk:proviso}
This result can be found in Malham \& Wiese (2009) for $\Ab=\{1,\ldots,d\}$.
It extends to $\Ab=\{0,1,\ldots,d\}$ using Lemma~\ref{lemma:rev},
though under the proviso that we grade
according to word length. This means that we endeavour to include
some multiple integrals in our stochastic integrator involving the 
drift `0' index that we would not ordinarily include if we were 
grading according to variance (of the multiple Wiener integrals).
However we take the perspective here that there are
far fewer of these terms at each grading according to length,
and the computational effort associated with their simulation
is a small fraction of that overall. If we include them, 
we guarantee efficiency. See \S\ref{sec:conclu} for an example.
Of course if $\Ab=\{1,\ldots,d\}$ only, then the two notions of grading
coincide and this technicality is redundant.
\end{remark}
Some natural questions now arise. It is apparent that
the first key result in the argument above to prove 
efficiency was 
Lemma~\ref{lemma:rev}(4) showing that
the norm-square of the identity endomorphism 
decomposes into the sum of the norm-squares
of the sinhlog and coshlog endomorphisms. The 
second key result, to prove that the sinhlog integrator was
optimal when $n$ is odd, was Lemma~\ref{lemma:rev}(8).
First, with regard to efficiency and 
Lemma~\ref{lemma:rev}(4). Since there
is no immediate apparent difference, it would seem
we could choose coshlog as our integrator endomorphism.
Further, the result of Lemma~\ref{lemma:rev}(4) essentially
relied on the fact that $\la S, S\ra=\la\id,\id\ra$. However,
we also know from Lemma~\ref{lemma:rev}(2) that $\la|S|,|S|\ra=\la\id,\id\ra$.
Thus, in principle, we could also consider $\frac{1}{2}(\id-|S|)$
as an integrator (or indeed any pair $\frac{1}{2}(\id\pm X)$
for which $\la X, X\ra=\la\id,\id\ra$). 
Note that $\frac{1}{2}(\id-|S|)$ is the equivalent of
applying the sinhlog endomorphism at even orders and 
the coshlog endomorphism at odd orders. Second, with regard to 
optimal efficiency, it would also seem that 
$\frac{1}{2}(\id-|S|)$ would deliver optimality at both
even and odd orders as the linear term in $\epsilon$ above
would be zero. However, this is not the case.
The key is the relation between the error $R$
and the pre-remainder $Q$. 

We define the following class of endomorphisms, set
\begin{equation*}
f^\star(X;\epsilon)\coloneqq \tfrac12(X-\epsilon X^{\star(-1)}).
\end{equation*}
for any $\epsilon\in\mathbb R$. Then we see 
$f^\star(\id,+1)$ is the sinhlog endomorphism and
$f^\star(\id,-1)$ is the coshlog endomorphism.
Note that the compositional inverse of $f^\star(X;\epsilon)$ is 
\begin{equation*}
f^{-1}(X;\epsilon)=X+h^\star(X,\epsilon\,\nu),
\end{equation*}
where $h^\star=h^\star(X,Y)$ is the convolutional
square root of $X^{\star 2}+Y$ given in the introduction.
Our main result is as follows.

\begin{theorem}\label{theorem:main}
For every $\epsilon>0$ the class of integrators 
$f^\star(\id;\epsilon)$ is \emph{efficient}.
When $\epsilon=1$, the error of the integrator $f^\star(\id;1)$
realizes its smallest possible value compared to
the error of the corresponding stochastic Taylor integrator,
i.e.\/ if $R_\epsilon$ denotes the remainder of the integrator,
we have that the mean-square excess
\begin{equation*}
\|(\id-E)\circ\id\|^2-\|(\id-E)\circ R_\epsilon\|^2
\end{equation*}
is positive and maximized at $\epsilon=1$.
Thus a strong stochastic integrator based on the sinhlog
endomorphism is \emph{optimally efficient} 
within this class.
\end{theorem}
\begin{proof}
We prove the theorem in four steps.
Note that we have for any $\epsilon\in\mathbb R$:
\begin{equation*}
f^\star(\id;\epsilon)=\tfrac12(1-\epsilon)\nu+\tfrac12(1+\epsilon)J
-\tfrac12\epsilon\bigl(J^{\star 2}-J^{\star 3}+\cdots\bigr).
\end{equation*}
The integrator of interest of strong order $n/2$ is
$P\coloneqq\ind_{\Sb_{\leqslant n}}\circ f^\star(\id;\epsilon)$ 
and given by
\begin{equation*}
P=\Bigl(\tfrac12(1-\epsilon)\nu+\tfrac12(1+\epsilon)J
-\tfrac12\epsilon\bigl(J^{\star 2}-J^{\star 3}
+\cdots+(-1)^nJ^{\star n}\bigr)\Bigr)\circ\ind_{\Sb_{\leqslant n}}.
\end{equation*}
The pre-remainder is 
$Q=f^\star(\id;\epsilon)\circ\ind_{\Sb_{\geqslant n+1}}$.
Further, the approximation 
$\widehat{\id}\coloneqq f^{-1}\circ P$
results in an error given by 
\begin{align*}
R=&\;\id-\widehat{\id}\\
=&\;f^{-1}\circ (P+Q)-f^{-1}\circ P\\
=&\;Q+h^\star(P+Q,\epsilon\,\nu)-h^\star(P,\epsilon\,\nu).
\end{align*}
Our goal now is to carefully examine this relationship between
$R$ and $Q$, and in particular, the difference on the right shown.
\smallskip

\emph{Step~1.} 
A thorough understanding of the function $h^\star$ is thus required,
and we motivate our analysis by considering the corresponding 
real-valued function $h\colon\mathbb R^2\to\mathbb R$ given by
\begin{equation*}
h\colon(x,y)\mapsto(x^2+y)^{1/2}.
\end{equation*}
The function $h=h(x,y)$ is real-analytic for $y>-x^2$, 
zero on $y=-x^2$ where its tangent plane is orthogonal
to the $(x,y)$-plane, and complex-valued elsewhere. 
Hence we can choose a point inside the open region $y>-x^2$
about which the Taylor series for $h=h(x,y)$
has a non-zero radius of convergence. In particular
for $y>-x^2$, the difference $h(x+q,y)-h(x,y)$ at 
leading order in $q$ is given by
\begin{equation*}
h(x+q,y)-h(x,y)=\frac{x}{(x^2+y)^{1/2}}\cdot q+\cdots.
\end{equation*}
\smallskip

\emph{Step~2.} By analogy with Step~1 we see in the 
convolution algebra we can expand
\begin{align*}
R=&\;Q+h^\star(P+Q,\epsilon\,\nu)-h^\star(P,\epsilon\,\nu)\\
=&\;Q+\bigl((P+Q)^{\star2}+\epsilon\,\nu\bigr)^{\star(1/2)}-
\bigl(P^{\star2}+\epsilon\,\nu\bigr)^{\star(1/2)}\\
=&\;Q+\bigl(P^{\star2}+\epsilon\,\nu\bigr)^{\star(1/2)}\\
&\;\qquad\star\biggl(\Bigl(\nu+\bigl(P^{\star2}+\epsilon\,\nu\bigr)^{\star(-1)}
\star\bigl(P\star Q+Q\star P+Q^{\star2}\bigr)\Bigr)^{\star(1/2)}-\nu\biggr)\\
=&\;Q+\tfrac12(P^{\star2}+\epsilon\,\nu)^{\star(-1/2)}\star\Bigl(\bigl(P\star Q+Q\star P\bigr)
+\mathcal O\bigl(Q^{\star2}\bigr)\Bigr).
\end{align*}
Further we note that at leading order 
\begin{equation*}
P=\tfrac12(1-\epsilon)\,\nu+\tfrac12(1+\epsilon)\,J+\mathcal O(J^{\star2}),
\end{equation*}
and thus for all $\epsilon>-1$ we have
\begin{equation*}
(P^{\star2}+\epsilon\,\nu)^{\star(-1/2)}=\bigl(\tfrac12(1+\epsilon)\bigr)^{-1}\,\nu+\mathcal O(J).
\end{equation*}
Substituting these expressions into that for $R$ above, we get at leading order:
\begin{align*}
R=&\;Q+\frac{1-\epsilon}{1+\epsilon}Q+\mathcal O(J\star Q)\\
=&\;\frac{2}{1+\epsilon}Q\circ\ind_{\Sb_{n+1}}.
\end{align*}
\smallskip

\emph{Step~3.} Let us now compare the mean-square error of $P$ to
the corresponding stochastic Taylor integrator $\id\circ\ind_{\Sb_{\leqslant n}}$.
We see that on $\Sb_{n+1}$ we have
{\allowdisplaybreaks
\begin{align*}
\|\id\|^2=&\;\la R+\id-R, R+\id-R\ra\\
=&\;\|R\|^2+2\la R,\id-R\ra+\la\id-R,\id-R\ra\\
=&\; \|R\|^2+2\biggl\langle\id-\frac{1}{1+\epsilon}(\id-\epsilon S),
\frac{1}{1+\epsilon}(\id-\epsilon S)\biggr\rangle\\
&\;\qquad\,\,+\biggl\langle\id-\frac{1}{1+\epsilon}(\id-\epsilon S),
\id-\frac{1}{1+\epsilon}(\id-\epsilon S)\biggr\rangle\\
=&\;\|R\|^2+\|\id\|^2-\frac{1}{(1+\epsilon)^2}\bigl(\|\id\|^2
-2\epsilon\la\id,S\ra+\epsilon^2\|S\|^2\bigr)\\
=&\;\|R\|^2+\frac{2\epsilon}{(1+\epsilon)^2}\bigl(\|\id\|^2+\la\id,S\ra\bigr).
\end{align*}}
If we include the expectations of the leading order terms 
in the integrator (again cf.\/ Remark~\ref{remark:expect}) 
then we have 
\begin{equation*}
(\id-E)\circ R=\frac{2}{1+\epsilon}(\id-E)\circ Q\circ\ind_{\Sb_{n+1}}.
\end{equation*}
The comparison above becomes 
\begin{multline*}
\|(\id-E)\circ\id\|^2=\|(\id-E)\circ R\|^2\\
+\frac{2\epsilon}{(1+\epsilon)^2}
\Bigl(\bigl\|(\id-E)\circ\id\bigr\|^2+\bigl\langle(\id-E)\circ\id,(\id-E)\circ S\bigr\rangle\Bigr).
\end{multline*}
\smallskip

\emph{Step 4.} We see the mean-square excess is 
\begin{equation*}
\frac{2\epsilon}{(1+\epsilon)^2}
\Bigl(\bigl\|(\id-E)\circ\id\bigr\|^2
+\bigl\langle(\id-E)\circ\id,(\id-E)\circ S\bigr\rangle\Bigr).
\end{equation*}
Note for any $X,Y\in\Hb$ we have
$\bigl|\la X,Y\ra\bigr|\leqslant\|X\|\,\|Y\|$.
Using Lemma~\ref{lemma:rev} we see that
\begin{align*}
\bigl\|(\id-E)\circ|S|\bigr\|^2
&=\bigl\langle |S|,|S|\bigr\rangle-\bigl\langle E\circ|S|,E\circ|S|\bigr\rangle\\
&=\la\id,\id\ra-\la E\circ\id,E\circ\id\ra\\
&=\bigl\|(\id-E)\circ\id\bigr\|^2.
\end{align*}
Hence we observe that
$\bigl\langle(\id-E)\circ\id,(\id-E)\circ|S|\bigr\rangle
\leqslant\bigl\|(\id-E)\circ\id\bigr\|^2$,
and thus 
\begin{equation*}
0<\Bigl|\bigl\langle(\id-E)\circ\id,(\id-E)\circ S\bigr\rangle\Bigr|
\leqslant\bigl\|(\id-E)\circ\id\bigr\|^2.
\end{equation*}
Note all our statements above hold for any $\mathsf V$.
Hence for all $\epsilon>0$, the mean-square excess is positive 
and thus $P$ is a uniformly accurate integrator, thus establishing 
the first statement of the theorem. When $\epsilon=0$
the mean-square excess is zero as expected. For $\epsilon<0$
it is negative. Further we see that for $\epsilon>0$ 
the mean-square excess is largest when $\epsilon=1$,
corresponding to the sinhlog endomorphism. 
We have thus established the second statement of the theorem and the
proof is complete.
\end{proof}
\smallskip

Note when $\epsilon=-1$ the mean-square excess
in Step~4 of the proof above is undefined. 
We can explain this as follows. 
The first equation in Step~2 becomes 
\begin{equation*}
R=Q+\tfrac12(P^{\star2}-\nu)^{\star(-1/2)}\star\bigl(P\star Q+Q\star P\bigr)
+\mathcal O\bigl(Q^{\star2}\bigr).
\end{equation*}
However now the leading order behaviour in $P$ is
$P=\nu+\tfrac12J^{\star2}+\mathcal O\bigl(J^{\star3}\bigr)$.
If we denote $\hat P=P-\nu$ then we see that
\begin{align*}
(P^{\star2}-\nu)^{\star(1/2)}
=&\;\bigl((\nu+\hat P)^{\star2}-\nu\bigr)^{\star(1/2)}\\
=&\;\sqrt{2}\,\hat P^{\star(1/2)}\star\bigl(\nu+\tfrac12\hat P\bigr)^{\star(1/2)}\\
=&\;\sqrt{2}\,\hat P^{\star(1/2)}\star\Bigl(\nu+\tfrac14\hat P
+\mathcal O\bigl(\hat P^{\star 2}\bigr)\Bigr).
\end{align*}
Proceeding formally, we substitute these last two expansions 
into the expression for $R$ above. Retaining leading order terms we find
\begin{equation*}
R=Q+\bigl(J^{\star(-1)}\circ\ind_{\Sb_{\leqslant n}}\bigr)\star Q+\mathcal O\bigl(J\star Q\bigr).
\end{equation*}
In other words, the term in the error $R$ corresponding to the difference
of $h^\star$ evaluated at $P+Q$ and $P$ generates the term 
$\bigl(J^{\star(-1)}\circ\ind_{\Sb_{\leqslant n}}\bigr)\star Q$. The inverse
of $J$ is not an element of the group $\mathcal G$. It does have the formal 
expansion $J^{\star(-1)}=S+S^{\star2}+S^{\star3}+\cdots$, but this 
does not have a finite evaluation on the empty word. Hence on
$\Sb_{n+1}$, the term $\bigl(J^{\star(-1)}\circ\ind_{\Sb_{\leqslant n}}\bigr)\star Q$
is not finite. However it is finite on $\Sb_{n}$---the
inverse contracts the number of deconcatenations---but this
now introduces terms in the remainder of the same order as those we
retain in the integrator, i.e.\/ we have a form of order reduction.

\section{Concluding remarks}\label{sec:conclu}
We established the sinhlog integrator is
optimally efficient when we grade according
to word length. For example, the sinhlog integrator
specified by $\Gc(w)\leqslant 2$ on the computation
interval $[t_n,t_{n+1}]$ is given by
\begin{equation*}
y_{n+1}=\shl^{-1}(\hat\sigma_{n,n+1})\circ y_{n},
\end{equation*}
where
\begin{equation*}
\hat\sigma_{n,n+1}=\sum_{i=0}^dJ_i(t_n,t_{n+1})V_i
+\sum_{i,j=0}^d\tfrac12\bigl(J_{ij}-J_{ji}\bigr)(t_n,t_{n+1})V_{ij}.
\end{equation*}
Whatever the vector fields are, this is guaranteed to be more
accurate than
\begin{equation*}
y_{n+1}=y_{n}+\sum_{i=0}^dJ_i(t_n,t_{n+1})V_i(y_n)
+\sum_{i,j=0}^dJ_{ij}(t_n,t_{n+1})V_{ij}(y_n),
\end{equation*}
which is the corresponding integrator based on the stochastic
Taylor expansion according to the grading $\Gc(w)\leqslant 2$.
(A thorough numerical investigation confirming this conclusion 
in the diffusion-only case is presented in Malham \& Wiese (2009).
Further, Lord \textit{et al.\/ } (2008) demonstrate that high accuracy, 
higher order stochastic integrators can deliver greater accuracy
for a given computational effort, than the Euler--Maruyama scheme,
despite the cost associated with simulating the multiple Wiener integrals
included.) First we note that at this order, $\hat\sigma_{n,n+1}$ 
is also the form of the exponential Lie series for $\Gc(w)\leqslant 2$
(this is not true at any higher orders) which is not
an efficient integrator. However at this order, the two integrators
are distinguished when we compute their inverse endomorphisms
to generate the corresponding approximations.
Second, the inverse of the sinhlog endomorphism is 
$\shl^{-1}(\sigma)=\sigma+(\id+\sigma^2)^{1/2}$. 
If the vector fields are linear, $\hat\sigma_{n,n+1}$ is
a matrix and we can construct $\shl^{-1}(\sigma)$
by computing the matrix square-root. If the 
vector fields are nonlinear we can expand the 
square root to sufficiently high degree terms 
(see Malham \& Wiese 2009, Remark~5.3).
Third, the stochastic Taylor based scheme above
is a modification of the Milstein method where 
we additionally include terms involving $J_{i0}$,
$J_{0i}$ and $J_{00}$ for $i=1,\ldots,d$.

We considered here a comparison under the mean-square measure
of the local errors. How the local errors accumulate to contribute to
the global error was considered in Malham \& Wiese (2009), where
it was shown that the local error comparison transfers to the 
global error.

If we know further structure in the stochastic differential system
concerned, then the class of endomorphisms that are efficient
will widen. For example, if the diffusion vector fields 
commute with themselves, but not with the drift vector field, 
then we can deduce that the exponential Lie series
is (trivially) an efficient integrator, again, under the
proviso we grade according to word length. This is 
because in the remainder at leading order, the largest
terms according to root-mean-square scaling with respect to stepsize,
will involve words with non-zero letters. However these will  
not in fact be present as the diffusion vector fields commute.

Our results have only relied on the symmetry properties of
the expectation map (see Lemma~\ref{lemma:rev}) and that 
it realizes positive values
on words with a scaling according to word length. In particular 
they do not depend on the coefficients explicitly. Hence the 
sinhlog integrator is optimally efficient within
the whole class of possible coefficients. In a slightly
different direction, our result also holds for weak approximations
of the multiple integrals, as long as the $L^2$-norm of (sums of)
integrals is preserved. We assume here, that the error is still
being measured in the mean-square sense.

Though in general the Eulerian idempotent is not efficient, 
it is a natural object in the construction 
of geometric integrators. Recently the Dynkin and Klyachko 
Lie algebra idempotents have proved to be important in the 
context of geometric integration, see Chapoton (2009), 
Patras \& Reutenauer (2002), Munthe--Kaas \& Wright (2007), 
and Lundervold \& Munthe--Kaas (2011a,b). Interesting
questions here besides the extension of their use to stochastic
differential equations is which of these projectors generates
a more accurate and efficient geometric integrator than the other?

\begin{acknowledgements}
All the authors would like to thank the London Mathematical
Society with Scheme 4, as well as the Edinburgh
Mathematical Society, for support for a visit by KEF, AL and HMK
to Heriot--Watt in December 2009 when a large part of the research
herein was initiated. We also thank ICMAT and CSIC for 
financial support during the research 
trimester Combinatorics and Control in Madrid in 2010. During this
research trimester AL received financial support from the NILS
Mobility Project, UCM-EEA Abel Extraordinary Chair 2010.
We thank Steven Gray, Luis Duffaut Espinosa and Domonique Manchon
for lively and useful conversations during the trimester.
AW and SJAM would like to thank the 
Institute for Mathematics and its Applications and
the Centre for Numerical Analysis and Intelligent Software,
respectively, for financial support.
SJAM and AW thank Charles Curry for 
useful conversations. Lastly we are also
grateful to the anonymous referees whose comments
and suggestions helped to significantly improve the original paper.
\end{acknowledgements}

\end{document}